%% file: Tangle-Replacements-V1.tex
\documentclass[a4,11pt]{paper}
\usepackage{enumitem,color,amssymb,bm,amsmath,stmaryrd,tikz-cd,mathtools,hyperref}
\usepackage[english]{babel}  
\usepackage{graphicx}
\usepackage{lipsum}
\usepackage{xcolor}
\usepackage{booktabs}
\usepackage{scalerel}
\usepackage{ntheorem}
{\theoremstyle{empty}
	}
\usepackage[top=1in, left=1in, right=1in, bottom=0.8in]{geometry}
\newcommand{\R}{\mathbb{R}}
\newcommand{\Z}{\mathbb{Z}}

\newcommand{\F}{\mathbb{F}}

\newcommand{\A}{{\mathbb{A}}}

\newcommand{\zs}{{\mathbf{z}}}
\newcommand{\ws}{{\mathbf{w}}}

\newcommand{\var}{\mathsf{u}}
\newcommand{\varv}{\mathsf{v}}

\newcommand{\alphas}{\mbox{\boldmath$\alpha$}}
\newcommand{\betas}{\mbox{\boldmath$\beta$}}
\newcommand{\gammas}{\mbox{\boldmath$\gamma$}}
\newcommand{\deltas}{\mbox{\boldmath$\delta$}}
\newcommand{\Sig}{\Sigma}

\newcommand{\Order}{\mathrm{Ord}}
\newcommand{\CFK}{\mathsf{CFL}}
\newcommand{\HFK}{\mathsf{HFL}}
\newcommand{\Torsion}{\mathrm{Tor}}

\subsectionfont{\large\sf\bfseries\color{black!50!blue}} 
\sectionfont{\Large\sf\bfseries\color{black!50!blue}} 
\titlefont{\LARGE\sf\bfseries\color{black!50!blue}} 
\numberwithin{equation}{section}
\input{commands-nmd}

\begin{document}
	\title{Tangle replacements and  knot Floer homology torsions}%
	\author{Eaman Eftekhary}%
	\institution{\sf{School of Mathematics, Institute for Research in Fundamental Sciences (IPM), Tehran, Iran}}
	\maketitle
	\begin{abstract}
	We show that the torsion order $\Order(K)$ of a knot $K$ in knot Floer homology  gives a lower bound on the minimum number $n$ such that an oriented $(n+1)$-tangle replacement unknots $K$. 
	This generalizes earlier results by Alishahi and the author and by Juh\'asz, Miller and Zemke, that $\Order(K)$ is a lower bound for both the unknotting number $u(K)$ and  for $\br(K)-1$, where  $\br(K)$ denotes the bridge index of $K$.
	\end{abstract}

\section{Introduction} 
Alishahi and the author \cite{AE-2}, and independently Zemke \cite{Zemke-1}, associated Heegaard-Floer cobordisms maps to decorated cobordisms connecting pointed links (also see \cite{Ef-TQFT}). As a byproduct, lower  bounds for  the unknotting number $u(K)$ of a knot $K$, including the torsion order $\Order(K)$ of the minus knot Floer homology group $\HFK^-(K)$, were introduced by Alishahi and the author \cite{AE-cobordisms-v1,AE-unknotting}.   
It was further shown by the author that $\Order(K)$ is in fact a lower bound for the minimum number $u_q(K)$ of proper rational replacements needed for unknotting $K$ \cite{Ef-RTR}.  Juh\'asz, Miller and Zemke proved that $\Order(K)$ is also a lower bound for $\br(K)-1$, where $\br(K)$ denotes the bridge index of  $K$ \cite{JMZ}.  It is then natural to ask weather there is a common generalization of the results of \cite{Ef-RTR} and \cite{JMZ} on geometric features of a knot $K$ bounded below by $\Order(K)$?   The goal of this short paper is to provide an answer to the latter question. \\

To state our result, let us fix some notation and review some definitions.
For $i\in\Z^{\geq 0}$, let $J_i$ denote the intersection of the line $\{(0,\frac{i-1}{i})\}\times \R$ with the standard $3$-ball $B^3$  and set 
\begin{align*}
I_n=\cup_{i=0}^nJ_i,\quad 
\partial^\pm J_i=\Big(0,\frac{i-1}{i},\pm \frac{\sqrt{2i-1}}{i}\Big)\quad\text{and}\quad
\partial^\pm I_n=\cup_{i=0}^n\partial^\pm J_n.	
\end{align*}	
A {\emph{trivial $(n+1)$-tangle}} is the diffeomorphic image $(B, T)$ of $(B^3,I_n)$ in $\R^3$.  $\partial^\pm T$ is defined as the image of $\partial^\pm I_n$.
 If $(B,T)$ is the intersection of $B$ with a link $K$, a {\emph{tangle replacement}}
in $B$ is the replacement of $T$ with $f(T)$, where $f:B\ra B$ is a diffeomorphism keeping $\partial T$ invariant. A $2$-tangle replacement  is called a {\emph{rational replacement}}. If $K$ is oriented so that the induced orientation of $T$ is  from $\partial^-T$ to $\partial^+T$ and $f$ keeps either of $\partial^\pm T$  invariant, the tangle replacement  is called  {\emph{oriented}}.   Define the {\emph{tangle distance}}  $d_t(K,K')$ as the minimum $n\in\Z^+$ such that an $(n+1)$-tangle replacement  changes $K$ to $K'$. The  oriented tangle distance  $d_{ot}(K,K')$ is defined similarly.  \\

For an oriented pointed link $\Kcal=(K,\ws,\zs)$, let $\CFK(\Kcal)$ denote the corresponding link chain complex constructed from a Heegaard diagram $(\Sigma,\alphas,\betas,\ws,\zs)$, 
which is generated over $\F[\var,\varv]$ by the intersection points in $\Ta\cap\Tb$ (where $\F=\Z/2$) and is equipped with the differential  
\begin{align*}
	d:\CFK(\Kcal)\ra\CFK(\Kcal),\quad
	d(\x)&:=\sum_{\y\in\Ta\cap \Tb}\ \ \sum_{\substack{\phi\in\pi_2(\x,\y)\\ \mu(\phi)=1}}\#(\Mhatt(\phi))\var^{n_\ws(\phi)}\varv^{n_\zs(\phi)}\cdot \y,\quad\forall\ \x\in\Ta\cap\Tb.
\end{align*}
Following \cite{AE-1}, for a $\F[\var,\varv]$-algebra $\A$,  let $\HFK(\Kcal;\A)$ denote the homology of the chain complex 
\[\CFK(\Kcal;\A):=\CFK(\Kcal)\otimes_{\F[\var,\varv]}\A.\]
For $\A^-=\F[\var,\varv]/\langle \varv=0\rangle\simeq \F[\var]$ and $\A^==\F[\var,\varv]/\langle \varv=\var\rangle\simeq\F[\var]$ define 
\begin{align*}
\HFK^-(\Kcal):=\HFK(\Kcal;\A^-)\quad\text{and}\quad
\HFK^=(\Kcal):=\HFK(\Kcal;\A^=).	
\end{align*}	 
Note that $\HFK^-(\Kcal)$ is bi-graded (by the homological and the Alexander gradings) while $\HFK^=(\Kcal)$ is only equipped with a homological grading. An element $\x$ of $\HFK^-(\Kcal)$ or $\HFK^=(\Kcal)$ is called a torsion element of order $k$ if $\var^{k-1}\x\neq 0$, while $\var^k\x=0$. The torsion orders $\Order(K)$ and $\Order'(K)$, which do not depend on the decoration $\Kcal$ of $K$, are defined as the maximum order of a torsion element in $\HFK^-(\Kcal)$ and $\HFK^=(\Kcal)$, respectively.

\begin{thm}\label{thm:main-intro}
Given two links $K$ and $K'$ in $S^3$ 
\begin{align*}
d_{ot}(K,K')+\Order(K')\geq \Order(K)
\quad\text{and}\quad d_{t}(K,K')+\Order'(K')\geq \Order'(K). 	
\end{align*}		
\end{thm}	 

Rational unknotting is considered by Lines \cite{Lines} and McCoy \cite{McCoy}. If a tangle replacement changes $K$ to $K'$, surgery on a corresponding solid handlebody  changes the double cover $\Sigma(K)$ of $S^3$, branched along $K$, to $\Sigma(K')$. Thus, $d_t(K,K')$ is bounded below by the minimum genus $g(K,K')$ of a solid handlebody $S$ in $\Sigma(K)$ so that surgery on $S$ gives $\Sigma(K')$. This observation, called the {\emph{Montesinos trick}}, is used by McCoy and Zenter to study proper rational unknotting  \cite{MZ-rr}. 
Theorem~\ref{thm:main-intro} suggests comparing $g(K,U)$ and the torsion orders, where $U$ denotes the unknot.

\section{Proof of the main theorem}\label{sec:main-theorem}
Let $K\subset \R^3$ be an oriented knot or link and fix the marked points $\ws$ and $\zs$ on $K$ so that they have nonempty intersection with and alternate on every connected component of $K$. Let us further assume that $|\ws|=|\zs|=N+1$ for a positive integer $N$.
The  bi-graded chain complex $\CFK^-(\Kcal)$ associated with the pointed knot $\Kcal=(K,\ws,\zs)$ is then  chain homotopy equivalent to a chain complex $(C_\Kcal,d_\Kcal)$ with the property that $C_\Kcal$ is generated by the generators $\{\x_I\}_{I\subset \{1,\ldots,n\}}$, indexed by the subsets $I$ of $\{1,\ldots,N\}$, as well as the generators $\y_1,\ldots,\y_m,\z_1,\ldots,\z_m$, so that 
\begin{align*}
	d_\Kcal(\x_I)=0,\quad\forall\ I\subset\{1,\ldots,N\}\quad\text{and}\quad
	d_\Kcal(\y_j)=\z_j,\ \ d_\Kcal(\z_j)=0,\quad\forall\ j=1,\ldots,m.	
\end{align*}	
The bigrading of $\x_I$ is given by $(-|I|,-a_I)$ for the integers $a_I\in \Z$, while the bigrading of $\z_j$ is given by $(-h_j,-a_j)\in\Z^2$. It then follows that the bigrading of $\y_j$ is $(1-h_j,k_j-a_j)$. If this is the case, the homology group $\HFK^-(\Kcal)$ is given as
\begin{align*}
	\HFK^-(\Kcal)=\Big(\bigoplus_{I\subset \{1,\ldots,N\}}	\F[\var]\lbr |I|,a_I\rbr\Big)\oplus
	\Big(\bigoplus_{j=1}^m\frac{\F[\var]}{\var^{k_j}=1}\lbr h_j,a_j\rbr\Big)=:
	\Abb(\Kcal)\oplus\Torsion(\Kcal).
\end{align*}	
The torsion subgroup $\Torsion(\Kcal)$ is canonically associated with $\Kcal$ as a subgroup of $\HFK^-(\Kcal)$, while the torsion-free part $\Abb(\Kcal)$ is defined as the quotient $\HFK^-(\Kcal)/\Torsion(\Kcal)$. Therefore, the direct sum on the right-hand-side of the above equation is not canonical. Note that $\Abb(\Kcal)$ admits a distinguished top generator $\Theta_{\Kcal}$ (with respect to the homological grading). We will usually regard $\Theta_{\Kcal}$ as an element in $\HFK^-(\Kcal)$ which is well-defined up to torsions (i.e. elements in $\Torsion(\Kcal)$). \\

Let us further assume that the intersection of a ball $B\subset \R^3$ with an oriented link $K$ is the trivial $(n+1)$-tangle $T$ and that $\ws$ and $\zs$ do not intersect $B^\circ$, while $\partial^-T\subset\ws$ and $\partial^+T\subset \zs$. Thus every component $T_i$ of $T$ (for $i=0,1,\ldots,n$) connect a  point  $w_i\in\ws\cap\partial B$ to a point  $z_i\in\zs\cap\partial B$. 
We assume that the oriented tangle replacement in the following theorem is 
performed on $(B,T)$.

\begin{thm}\label{thm:PTR}
	If the pointed link $\Kcal'=(K',\ws,\zs)$ is obtained from the pointed link $\Kcal=(K,\ws,\zs)$ by an oriented $(n+1)$-tangle replacement, there are $\F[\var]$-homomorphisms 
	\begin{align*}
		\fmap_{\Kcal\ra \Kcal'}:\HFK^-(\Kcal)\lra \HFK^-(\Kcal')\quad\text{and}\quad	
		\fmap_{\Kcal'\ra \Kcal}:\HFK^-(\Kcal')\lra \HFK^-(\Kcal).	
	\end{align*}	
	so that for some $k,k'\in\{0,1,\ldots,n\}$ we have 
	\begin{align*}
		\fmap_{\Kcal'\ra \Kcal}\circ\fmap_{\Kcal\ra \Kcal'}=\var^{k}\cdot Id\quad\text{and}\quad
		\fmap_{\Kcal\ra \Kcal'}\circ\fmap_{\Kcal'\ra \Kcal}=\var^{k'}\cdot Id.	
	\end{align*}	
\end{thm}
\begin{proof}
Suppose that $\Kcal'$  is obtained from $\Kcal$ by replacing  $(B,T)$ with  $(B,T')$, as discussed above. Let $D_i$ and $D'_i$ denote disjoint small open disks on $\partial B$ containing  $w_i$ and $z_i$ on their boundary, respectively. We may then construct a link diagram for $\Kcal$ with the following properties (see \cite[Section 2]{Ef-RTR} for the construction in the case where $n=1$). The Heegaard surface $\Sig$ is the union of 
 \begin{align*}
 S=\partial B-\cup_{i=0}^n D_i-\cup_{i=0}^n D'_i	
 \end{align*}	
with another sub-surface (with boundary) which is denoted by $\Sig'$, so that 
 \begin{align*}
	\Sig'\cap{S}=\big(\cup_{i=0}^n \partial D_i\big)\cup\big(\cup_{i=0}^n \partial D'_i\big).	
\end{align*}	
The link diagram $H=(\Sig,\alphas,\betas,\ws,\zs)$ is constructed so that 
$\betas=\betas_0\cup\{\beta_1,\ldots,\beta_n\}$ with $\betas_0$ completely disjoint from ${S}$ and $\beta_1,\ldots,\beta_n\subset S^\circ$ cutting $S$ into $n+1$ connected components
$S_0,\ldots,S_n$ so that $S_i$ includes $w_i$ and $z_i$. We may then construct a link diagram for $\Kcal'$ of the form
\begin{align*}
H'=(\Sig,\alphas,\gammas=\gammas_0\cup\{\gamma_1,\ldots,\gamma_n\},\ws,\zs)	
\end{align*}	  
where $\gammas_0$ is obtained from $\betas_0$ by a small Hamiltonian isotopy away from the markings and $\gamma_1,\ldots,\gamma_n$ are simple closed disjoint curves on $S$ cutting it into connected components $S'_0,\ldots,S'_n$ so that $w_i,z_\sigma(i)\in S'_i$ for $i=0,\ldots,n$ and a permutation $\sigma\in S_{n+1}$ of $\{0,\ldots,n\}$.\\

The Heegaard triple $(\Sig,\alphas,\betas,\gammas,\ws,\zs)$ determines a  cobordism from the split union of $K$ and the $(n+1)$-bridge link $L=T\cup T'$ to $K'$, and a decoration of the aforementioned cobordism. Here, $L$ is regarded as a link in the sphere obtained by gluing two copies of $B$ along their sphere boundaries from the union of $T$ in one copy of $B$ and $T'$ in the second copy of $B$.  
If we set $\ws_0=\{w_0,\ldots,w_n\}$ and $\zs_0=\{z_0,\ldots,z_n\}$, and $\Lcal=(L,\ws_0,\zs_0)$, 
the aforementioned Heegaard triple determines a map
\begin{align*}
	\fmap:\HFK^-(\Kcal)\otimes_{\F[\var]}\HFK^-(\Lcal)\lra \HFK^-(\Kcal').
\end{align*}
Given $\y\in\HFK^-(\Lcal)$, we denote the map $\fmap(\cdot\otimes\y):\HFK^-(\Kcal)\lra \HFK^-(\Kcal')$ by $\fmap_\y$.
Similarly, we may use the Heegaard triple $(\Sig,\alphas,\gammas,\betas,\ws,\zs)$ to obtain a second $\F[\var]$-homomorphism
\begin{align*}
	\gmap:\HFK^-(\Kcal')\otimes_{\F[\var]}\HFK^-(\overline{\Lcal})\lra \HFK^-(\Kcal),
\end{align*}
where $\overline{\Lcal}=(\overline{L},\ws_0,\zs_0)$ and $\overline{L}=T'\cup T$ denotes the mirror of $L$ obtained by gluing the copy of $T'$ in the first copy of $B$ to the copy of $T$ in the second copy of $B$. Again, for every $\z\in\HFK^-(\overline{\Lcal})$ we obtain a corresponding $\F[\var]$-homomorphism
\begin{align*}
\gmap_\z=\gmap(\cdot\otimes\z):\HFK^-(\Kcal')\lra \HFK^-(\Kcal).	
\end{align*}	
Let $\deltas$ denote a copy of $\betas$ which is moved by a small Hamiltonian isotopy. Using the diagrams $(\Sig,\alphas,\betas,\deltas;\ws,\zs)$ and  $(\Sig,\betas,\gammas,\deltas,\ws,\zs)$, one obtains the  $\F[\var]$-homomorphisms
\begin{equation}\label{eq:maps-from-quadruple}
\begin{split}	
	&\jmap:\HFK^-(\Kcal)\otimes_{\F[\var]} \HFK^-(\Ucal)
	\lra \HFK^-(\Kcal)\quad\text{and}\\
	&\hmap:\HFK^-(\Lcal)\otimes_{\F[\var]}\HFK^-(\overline{\Lcal})\lra \HFK^-(\Ucal)=	\left(\F[\var]\oplus \F[\var]\lbr1,0\rbr\right)^{n},
\end{split}
\end{equation}	
where $\Ucal=(U,\ws_0,\zs_0)$ and $U$ is the $(n+1)$-component unlink with component $U_i$ hosting the marked points $w_i$ and $z_i$. On the other hand, the chain homotopy map coming from the Heegaard quadruple $(\Sig,\alphas,\betas,\gammas,\deltas;\ws,\zs)$, and an analysis of the degenerations of all holomorphic rectangles of index $0$ for the aforementioned Heegaard quadruple may be used to show that 
\begin{align}\label{eq:composition-of-maps}
&\gmap_\z(\fmap_\y(\x))=\jmap(\x\otimes\hmap(\y\otimes\z)),\quad
\forall\ \x\in\HFK^-(\Kcal)
,\ \y\in\HFK^-(\Lcal),\ \z\in\HFK^-(\overline{\Lcal}).
\end{align}	
If $\Theta=\Theta_{\Ucal}$ denotes the top generator  in $\HFK^-(\Ucal)=(\F[\var]\oplus \F[\var]\lbr1,0\rbr)^n$,
\begin{align*}
 \jmap(\cdot\otimes\Theta):\HFK^-(\Kcal)\ra \HFK^-(\Kcal)
\end{align*}
is the identity. The map $\hmap$ from (\ref{eq:maps-from-quadruple}) corresponds to a decoration of the ribbon cobordism from $L\#\overline{L}$ to the unlink $U$, as in \cite[Section 1.6]{JMZ}.  Here, in forming the connected sum $L\#\overline{L}$ we connect a point on a component in $L$ to a corresponding point in $\overline{L}$ by the connected sum band. Since two trivial $(n+1)$-tangles are glued together to form $L$, it follows that $L$ is an $(n+1)$-bridge link.
 The $(n+1)$-bridge presentation of $L$ gives a ribbon cobordism from $L\#\overline{L}$ to $U$ which is constructed by attaching $n$ bands to it, as discussed in \cite[Section 1.6]{JMZ} and illustrated in Figure~\ref{fig:n-band-cobordism} in a simple example. Therefore, \cite[Proposition 4.1]{JMZ} or the argument of \cite[Lemma 2.1]{AE-unknotting} imply that the aforementioned cobordism gives the maps
\begin{align*}
	&\hmap:\HFK^-(\Lcal)\otimes_{\F[\var]}\HFK^-(\overline{\Lcal})\lra \HFK^-(\Ucal)=	\left(\F[\var]\oplus \F[\var]\lbr1,0\rbr\right)^{n}\quad\text{and}\\
	&\hmap':\HFK^-(\Ucal)=	\left(\F[\var]\oplus \F[\var]\lbr1,0\rbr\right)^{n}\lra \HFK^-(\Lcal)\otimes_{\F[\var]}\HFK^-(\overline{\Lcal}),
\end{align*}	
so that $\hmap\circ\hmap'$ and $\hmap'\circ\hmap$ are both equal to $\var^n\cdot Id$. Since  $\Image(\hmap)$ is in a torsion-free module,  $\hmap(\y\otimes\z)=0$ if either of $\y$ or $\z$ is a torsion element. In particular, $\hmap$ factors through $\Abb(\Lcal)\otimes_{\F[\var]}\Abb(\overline{\Lcal})$, and 
\begin{align}\label{eq:composition-thetas}
\hmap\left(\Theta_{\Lcal}\otimes\Theta_{\overline{\Lcal}}\right)=\var^{k_{L}}\cdot \Theta \quad\text{for some }k_L\in\{0,1,\ldots,n\}.	
\end{align}	  

\begin{figure}
	\def\svgwidth{0.65\textwidth}
	{\small{
			\begin{center}
				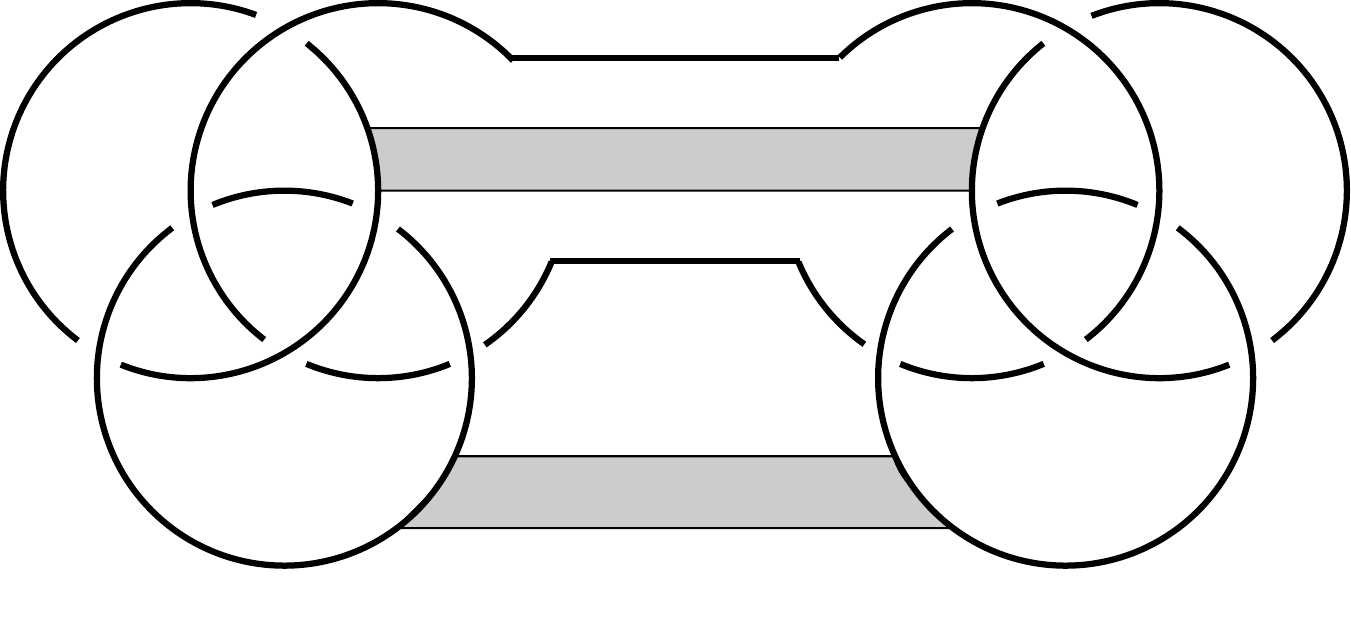 
	\end{center}}}
	\caption{\label{fig:n-band-cobordism} If $L$ is an  $(n+1)$-bridge link, attaching $n$ bands to $L\#\overline{L}$ gives an $(n+1)$ component unlink. This is illustrated for a $3$-bridge $3$-component link $L$.}
\end{figure} 

Although this is not relevant for the proof, note that $k_L$ only depends on $L$, and not the decoration $\Lcal$ of it. In particular,  if $\y\in\HFK^-(\Lcal)$ projects to $\Theta_{\Lcal}$ in $\Abb(\Lcal)$ and $\z\in\HFK^-(\overline{\Lcal})$ projects to $\Theta_{\overline{\Lcal}}$ in $\Abb(\overline{\Lcal})$, it follows from (\ref{eq:composition-of-maps}) and (\ref{eq:composition-thetas}) that 
\begin{align}\label{eq:composition-final}
&\gmap_\z(\fmap_\y(\x))=\var^{k_L}\cdot \x,\quad
\forall\ \x\in\HFK^-(\Kcal).
\end{align}	  
A similar argument implies that $\fmap_\y\circ\gmap_\z$ is also of the form $\var^{k_{\overline{L}}}\cdot Id$. Therefore, to complete the proof of the theorem, it suffices to set $\fmap_{\Kcal\ra \Kcal'}=\fmap_\y$ and $\fmap_{\Kcal'\ra \Kcal}=\gmap_\z$. 
\end{proof}

A similar proof implies the following theorem for general tangle replacements.

\begin{thm}\label{thm:TR}
	If the pointed link $\Kcal'=(K',\ws,\zs)$ is obtained from the pointed link $\Kcal=(K,\ws,\zs)$ by an $(n+1)$- tangle replacement, there are $\F[\var]$-homomorphisms 
	\begin{align*}
		\gmap_{\Kcal\ra \Kcal'}:\HFK^=(\Kcal)\lra \HFK^=(\Kcal')\quad\text{and}\quad	
		\gmap_{\Kcal'\ra \Kcal}:\HFK^=(\Kcal')\lra \HFK^=(\Kcal).	
	\end{align*}	
	so that for some $l,l'\in\{0,1,\ldots,n\}$ we have 
	\begin{align*}
		\gmap_{\Kcal'\ra \Kcal}\circ\gmap_{\Kcal\ra \Kcal'}=\var^{l}\cdot Id\quad\text{and}\quad
		\gmap_{\Kcal\ra \Kcal'}\circ\gmap_{\Kcal'\ra \Kcal}=\var^{l'}\cdot Id.	
	\end{align*}	
\end{thm}

Theorem~\ref{thm:main-intro} now follows as a direct
 corollary of Theorem~\ref{thm:PTR} and Theorem~\ref{thm:TR}.

\section{Examples and computations}\label{sec:examples}
In this section, we examine the bounds constructed in the previous sections in a number of examples. 

\begin{ex}\label{ex:torus-knot}
Let $K=T_{p,q}$ be the $(p,q)$ torus knot. It is shown in \cite[Corollary 5.2]{Ef-RTR} that 
\begin{align*}
\Order(T_{p,q})=\min\{p,q\}-1\quad\text{and}\quad \Order'(T_{p,pk+1})=\left\lfloor\frac{p}{2}\right\rfloor.	
\end{align*}	
Therefore, if an oriented $n$-tangle replacement unknots $T_{p,q}$, then  $n\geq \min\{p,q\}$ and  if an $n$-tangle replacement unknots $T_{p,pk+1}$, then $n>\lfloor p/2\rfloor$.
\end{ex}

\begin{ex}
For $K=12n_{404}$ we have $\Order(K)=2=d_{ot}(K,U)$, where $U$ is the unknot.  As noted in \cite[Example 5.4]{AE-unknotting}, the unknotting number $u(K)\in\{2,3\}$ is not known.
\end{ex}

\begin{ex} By  \cite[Example 5.3]{Ef-RTR}, for the $(2,-1)$ cable and the $(2,-3)$ cable of $T_{2,3}$, which are denoted by $K=T_{2,3;2,-1}$ and $K'=T_{2,3;2,-3}$, respectively,
	 we have
\begin{align*}
\Order(K)=\Order(K')=2\quad\text{and}\quad \Order'(K)=\Order'(K')=1=d_t(K,U)=d_t(K',U).		
\end{align*}	
Note that a single $2$-tangle replacement  unknots the $(2,2k+1)$-cable of every knot. 
\end{ex}
\begin{ex}\label{ex:cables-of-knot}
	Fix a knot $K$, and the co-prime pairs of positive integers $(p_1,q_1),\ldots,(p_m,q_m)$. Let $L=K_{p_1,q_1;\ldots;p_m,q_m}$ denote the corresponding iterated cable of $K$. Hom, Lidman and Park have  shown \cite{HLP-cables} that if $K$ is non-trivial, the first inequality below is satisfied, while an explicit tangle replacement gives the second inequality:
	\begin{align*}
		\Order(L)\geq \max\{p_1\cdots p_m(\Order(K)-1)+1,p_1\cdots p_m\}\quad\text{and}\quad\Order'(L)\geq \left\lfloor\frac{p_1}{2}\right\rfloor\cdots\left\lfloor\frac {p_m}{2}\right\rfloor.	
	\end{align*}	
Moreover, if $K\neq T_{2,3}$ is an $L$-space knot, it also follows from \cite{HLP-cables} that
	\begin{align*}
\Order(L)\geq p_1\cdots p_m(\Order(K)+1)-1.	
	\end{align*}	
\end{ex}	
\begin{ex}\label{ex:ten-crossing-knots}
According to {\sf{Knotinfo}} tables \cite{Knotinfo} and the discussion in  \cite[Example 5.4]{Ef-RTR}, among the knots with at most $10$ crossings, all of them  have unknotting number $1$ or thin knot Floer homology, or at least their torsion order is $1$, except 
the knots in 
\[\Acal=\big\{8_{19},10_{124}, 10_{128}, 10_{139}, , 10_{152}, 10_{154}, 10_{161}\big\}.\]
If $K$ is either of the remaining $7$ knots in $\Acal, it$ may be changed to an alternating knot by a single oriented tangle replacement, while $\Order(K)=2$. In particular, 
if $K$ is either of the $L$-space knots  $8_{19}$ and $10_{124}$, $br(K)-1=\Order(K)=2$ and it follows from Example~\ref{ex:branched-double-cover} and  
\begin{align*}
	3p_1\cdots p_m-1\geq \Order(L)\geq \br(L)-1=
	3p_1\cdots p_m-1	
\end{align*}	  
that $d_{ot}(L)\geq  \Order(L)=\br(L)=3p_1\cdots p_m-1$. 
This claim is of course stronger than the claim that the bridge index of $L$ is   $3p_1\cdots p_m$.
\end{ex}
\begin{ex}\label{ex:branched-double-cover}
Given a knot $K\subset S^3$, let $\Sigma(K)$ denote the double cover of $S^3$ branched along $K$. If an $(n+1)$-tangle replacement unknots $K$, it follows that there is a handlebody $S$ of genus $n$ in $\Sigma(K)$ such that surgery on $S$ gives the standard sphere. Moreover, if  an $n$-tangle replacement changes $K$ to a quasi-alternating knot (which has think knot Floer homology and thus $\Order(K)=\Order'(K)=1$), it follows that there is a handlebody $S$ of genus $n$ in $\Sigma(K)$ such that surgery on $S$ gives and $L$-space. As mentioned in the introduction, this may be the source of many questions. In particular, if $n$ is the least integer such that there is a handlebody $S$ of genus $n$ in $\Sigma(K)$ so that surgery on $S$ gives and $L$-space, is it true that $\Order'(K)\geq n+1$?
\end{ex}	
\bibliographystyle{hamsalpha}

\end{document}

%% file: commands-nmd.tex
\newcommand{\br}{\mathrm{br}}

\newcommand{\Mhatt}{\widehat{\Mcal}}

\newcommand{\ra}{\rightarrow}
\newcommand{\lra}{\longrightarrow}

\newcommand{\lbr}{\llbracket}
\newcommand{\rbr}{\rrbracket}

\newcommand{\fmap}{\mathfrak{f}}
\newcommand{\gmap}{\mathfrak{g}}
\newcommand{\hmap}{\mathfrak{h}}

\newcommand{\jmap}{\mathfrak{j}}

\newcommand{\Abb}{\mathbb{A}}

\newcommand{\Acal}{\mathcal{A}}

\newcommand{\Kcal}{\mathcal{K}}
\newcommand{\Lcal}{\mathcal{L}}
\newcommand{\Mcal}{\mathcal{M}}

\newcommand{\Ucal}{\mathcal{U}}

\newcommand{\Image}{\mathrm{Im}}

\newtheorem{thm}{\bfseries\color{black!50!blue} Theorem}[section]

\newtheorem{ex}[thm]{\bfseries\color{black!50!blue} Example}

\def\endproof{\relax\ifmmode\expandafter\endproofmath\else
  \unskip\nobreak\hfil\penalty50\hskip.75em\hbox{}\nobreak\hfil\bull
  {\parfillskip=0pt \finalhyphendemerits=0 \bigbreak}\fi}
\def\endproofmath$${\eqno\bull$$\bigbreak}
\def\bull{\vbox{\hrule\hbox{\vrule\kern3pt\vbox{\kern6pt}\kern3pt
\vrule}\hrule}}

\newcommand\x{\mathbf x}
\newcommand\y{\mathbf y}

\newcommand\z{\mathbf z}

\newcommand\Ta{{\mathbb T}_{\alpha}}
\newcommand\Tb{{\mathbb T}_{\beta}}

\tikzset{commutative diagrams/.cd,
	mysymbol/.style={start anchor=center,end anchor=center,draw=none}
}

%% file: n-band-cobordism.pdf_tex
\begingroup%
  \makeatletter%
  \providecommand\color[2][]{%
    \errmessage{(Inkscape) Color is used for the text in Inkscape, but the package 'color.sty' is not loaded}%
    \renewcommand\color[2][]{}%
  }%
  \providecommand\transparent[1]{%
    \errmessage{(Inkscape) Transparency is used (non-zero) for the text in Inkscape, but the package 'transparent.sty' is not loaded}%
    \renewcommand\transparent[1]{}%
  }%
  \providecommand\rotatebox[2]{#2}%
  \newcommand*\fsize{\dimexpr\f@size pt\relax}%
  \newcommand*\lineheight[1]{\fontsize{\fsize}{#1\fsize}\selectfont}%
  \ifx\svgwidth\undefined%
    \setlength{\unitlength}{648.00709911bp}%
    \ifx\svgscale\undefined%
      \relax%
    \else%
      \setlength{\unitlength}{\unitlength * \real{\svgscale}}%
    \fi%
  \else%
    \setlength{\unitlength}{\svgwidth}%
  \fi%
  \global\let\svgwidth\undefined%
  \global\let\svgscale\undefined%
  \makeatother%
  \begin{picture}(1,0.47041275)%
    \lineheight{1}%
    \setlength\tabcolsep{0pt}%
    \put(0,0){\includegraphics[width=\unitlength,page=1]{n-band-cobordism.pdf}}%
    \put(0.18750494,0.00512388){\color[rgb]{0,0,0}\makebox(0,0)[t]{\lineheight{1.25}\smash{\begin{tabular}[t]{c}$L$\end{tabular}}}}%
    \put(0.85879403,0.00512388){\color[rgb]{0,0,0}\makebox(0,0)[t]{\lineheight{1.25}\smash{\begin{tabular}[t]{c}$\overline{L}$\end{tabular}}}}%
    \put(0.53703802,0.44493403){\color[rgb]{0,0,0}\makebox(0,0)[t]{\lineheight{1.25}\smash{\begin{tabular}[t]{c}$L\#\overline{L}$\end{tabular}}}}%
    \put(0.48842753,0.09771555){\color[rgb]{0,0,0}\makebox(0,0)[t]{\lineheight{1.25}\smash{\begin{tabular}[t]{c}first band\end{tabular}}}}%
    \put(0.48842752,0.34771292){\color[rgb]{0,0,0}\makebox(0,0)[t]{\lineheight{1.25}\smash{\begin{tabular}[t]{c}second band\end{tabular}}}}%
  \end{picture}%
\endgroup%